\documentclass[letterpaper, 10 pt, conference]{ieeeconf}  

\IEEEoverridecommandlockouts                              
\title{\LARGE \bf
Growing Controllable Networks \\via Whiskering and Submodular Optimization
}

\usepackage{graphicx} 
\usepackage{epsfig} 
\usepackage{amsmath} 
\usepackage{amssymb}  
\usepackage{color}
\usepackage{epstopdf}
\usepackage{textcomp}
\usepackage{mathtools}
\usepackage{mathrsfs}
\mathtoolsset{showonlyrefs}
\usepackage[belowskip=-10pt,aboveskip=5pt]{caption}
\newcommand{\startcompact}[1]{\par\vspace{-0.75em}\begin{#1}%
\allowdisplaybreaks\ignorespaces}

\newcommand{\stopcompact}[1]{\end{#1}\ignorespaces}

\author{Mathias Hudoba de Badyn and Mehran Mesbahi
\thanks{The research of the authors was supported by the U.S. Army Research
Laboratory and the U.S. Army Research Office under contract number W911NF-13-1-0340 and AFOSR grant  FA9550-16-1-0022. The authors are with the William E. Boeing Department of Aeronautics \& Astronautics, University of Washington, Seattle. Emails:
        {\tt\small \{hudomath, mesbahi\}@uw.edu}. \textcopyright~2016 IEEE}%
}

\begin{document} 
\newtheorem{lemma}{Lemma}
\newtheorem{theorem}[lemma]{Theorem}
\newtheorem{corollary}[lemma]{Corollary}
\newtheorem{definition}[lemma]{Definition}
\newtheorem{conjecture}[lemma]{Conjecture}
\newtheorem{remark}[lemma]{Remark}
\newtheorem{question}{Question}
\newcommand{\graph}{\mathcal{G}}
\newcommand{\redlap}{{L^{\{n-1\}}_\graph}}
\newcommand{\lap}{{L_\graph}}
\newcommand{\diri}{{\mathcal{A}_\graph}}
\newcommand{\Tr}{\text{Tr}}
\newcommand{\dirid}{\bar{\mathcal{A}}_\graph}

\maketitle
\thispagestyle{empty}
\pagestyle{empty}

\begin{abstract}
The topology of a network directly influences the behaviour and controllability of dynamical processes on that network.
Therefore, the design of network topologies is an important area of research when examining the control of distributed systems.
We discuss a method for growing networks known as whiskering, as well as generalizations of this process, and prove that they preserve controllability.
We then use techniques from submodular optimization to analyze optimization algorithms for adding new nodes to a network to optimize certain objectives, such as graph connectivity.

\end{abstract}

\section{INTRODUCTION}

A great amount of effort has recently been focused on understanding how the connection structure, or topology, of a network affects the behaviour or performance of a dynamical process on that network~\cite{Mesbahi2010}.
To that end, a natural question is how one can systematically construct a network topology such that a certain performance metric defined over that behaviour is satisfied.
A well-known method for constructing networks is of preferential attachment, where new nodes are attached to pre-existing nodes with a probability proportional to the degree of those nodes~\cite{Albert2002}.
The advantage of this method is that it produces networks with power-law degree distributions that resemble networks found in nature~\cite{Barabasi2009}.
Another method for growing networks uses Kolmogorov-Sinai entropy as a heuristic parameter for evolving networks~\cite{Demetrius2005}.

An area of recent focus is the study of controlling distributed systems~\cite{Mesbahi2010,Aguilar2014a,Boccaletti2006,Rahmani2009a,Chapman2015a,ORourke2015}. 
Some research pertains to how one may systematically construct a network that has favorable characteristics for consensus.
Ghosh and Boyd developed an algorithm to select connections between agents in a network to maximize the connectivity of the network \cite{Ghosh2006}.
Chapman and Mesbahi showed how to construct large networks from graph products of atomic networks and examined their controllability properties \cite{Chapman2013a, Chapman2014a}.
Yazicioglu and Egerstedt, and Abbas and Egerstedt worked on constructing networks for leader-follower selection~\cite{Yazicioglu2013, Abbas2011}.
Liu \emph{et al.} have discussed constructing graphs for scalable semi-supervised learning~\cite{Liu2010}.
When designing a network graph, there are several methods to achieve various performance characteristics.
For example, there has been recent work in using submodular optimization for picking input vectors~\cite{Cortesi2014,Summers2014}.
An excellent summary of submodular optimization applications to the control of networked systems is given in~\cite{Clark2016}.
Measures by which one may gauge network performance and specify network topologies have been suggested in~\cite{HudobaDeBadyn2015a,Siami2014}.

Whiskering is a process for growing graphs where at each iteration, a vertex and edge is connected to every node in the graph.
 In this paper, we discuss using this process, and generalizations of this process, to construct large graphs that are controllable. 
A process similar in spirit to whiskering are the `fractal' networks studied in the control-theoretic setting by Li \emph{et al.} in~\cite{Li2014a}

The contributions of this paper are as follows. 
We extend the use of submodular optimization in network science to problems involving adding \emph{nodes} to the network. 
We present a graph-growth method that preserves controllability of consensus on the graph, and provide relevant bounds on the network performance.
Lastly, we develop a graph-growth algorithm, and formulate convex optimization problems, which we then solve for specific test cases.

The paper is outlined as follows: in \S\ref{sec.prelim}, we present the terminology of graph theory, consensus dynamics and the relevant background and theorems in submodular optimization.
In \S\ref{sec.grow}, we discuss the graph whiskering process and its generalizations, and prove that they preserve controllability.
Optimization problems involving these processes are then formulated and discussed.
We implement algorithms to solve the optimization problems in \S\ref{sec.alg}.
The paper is concluded in \S\ref{sec.conc}, where future extensions of the work are discussed.

\section{MATHEMATICAL PRELIMINARIES}\label{sec.prelim}
This section consists of the relevant constructs we use later on for stating and proving our main results.

In this paper, we take the basic definitions for undirected graphs and matrices associated with them for granted; our graph theoretic notation is standard and can be found in~\cite{Mesbahi2010}.

 Let $m\in\mathbb{Z}_+$, and $[m]:=\{1,\dots,m\}$. 
We call a real-valued function $f:2^{[m]}\rightarrow\mathbb{R}$ \emph{nondecreasing} if for sets $J\subset K\subset [m]$, $f(J)\leq f(K)$.
The function $f$ is \emph{submodular} if for subsets $J,K\subset[m]$, we have that $  f(K) + f(J) \geq f(K\cup J) + f(K\cap J)$.
Furthermore, $f$ is \emph{nonincreasing} if $-f$ is nondecreasing, and $f$ is \emph{supermodular} if $-f$ is submodular.
$f$ is \emph{modular} if it is both supermodular and submodular.

A matrix $H$ is \emph{Hermitian} if $H=H^\dag$, where `$^\dag$' denotes the conjugate-transpose operation.
Let $\Lambda_E$ be a finite interval of $\mathbb{R}$, and thereby denote the set of $n\times n$ Hermitian matrices with eigenvalues contained in $\Lambda_E$ as $H_n(\Lambda_E)$.
We say that $f$ is \emph{operator monotone} on $H\left(\Lambda_E\right)$ if for all $n\geq 1$ and for all $A,B\in H_n(\Lambda_E)$, $  A\succeq B \implies f(A) \succeq f(B)$.
Lastly, let $A[K]$ denote the principal submatrix of $A$ obtained by deleting the rows and columns of $A$ corresponding to the elements in the set $[m]\setminus K$.

Recent works in submodular optimization have examined matrix functions (of say, $A$), such as the trace or the trace of powers of matrices, in the context of submodularity over a set $K\subseteq[n]$ on the principal submatrices $A[K]$~\cite{Friedland2013,Audenaert2010}.
We summarize the main results of these works, as well as a more specific result about submodularity of functions over principle sub-matrices of $M$-matrices.
\begin{theorem}[\cite{Friedland2013,Audenaert2010}]
  Let $f$ be a real continuous function on an interval $\Lambda_E$ of $\mathbb{R}$. 
Furthermore, let $f'$ be operator monotone on the interior of $H\left(\Lambda_E\right)$. Then, for all $A\in H_n(\Lambda_E)$, the map from $2^{[n]}\rightarrow\mathbb{R}$ given by $K\to \Tr f(A[K])$ is supermodular.
\end{theorem}
\begin{theorem}[\cite{Friedland2013}]\label{thm.supermod}
  Let $A$ be an $M$-matrix of size $n\times n$. Then, for all subsets $J,K\subset[n]$ we have that for $0\leq p\leq1$:
  \begin{align}
    \Tr A[K]^p + \Tr A[J]^p \geq \Tr A[K\cup J]^p + \Tr A[K\cap J]^p
  \end{align}
and for $1\leq p \leq 2$ and for $p<0$,
  \begin{align}
    \Tr A[K]^p + \Tr A[J]^p \leq \Tr A[K\cup J]^p + \Tr A[K\cap J]^p.
  \end{align}
In particular, the map $K\to \Tr A[K]^{-1}$ is supermodular if $A$ is an $M$-matrix.
\end{theorem}

\section{GRAPH GROWING}\label{sec.grow}
In this section, we introduce a method of growing graphs that preserves controllability, by adding a leaf to every node of the graph. 
We then generalize this process to adding more complicated structures.

\begin{figure}
  \centering
  \includegraphics[scale=0.4]{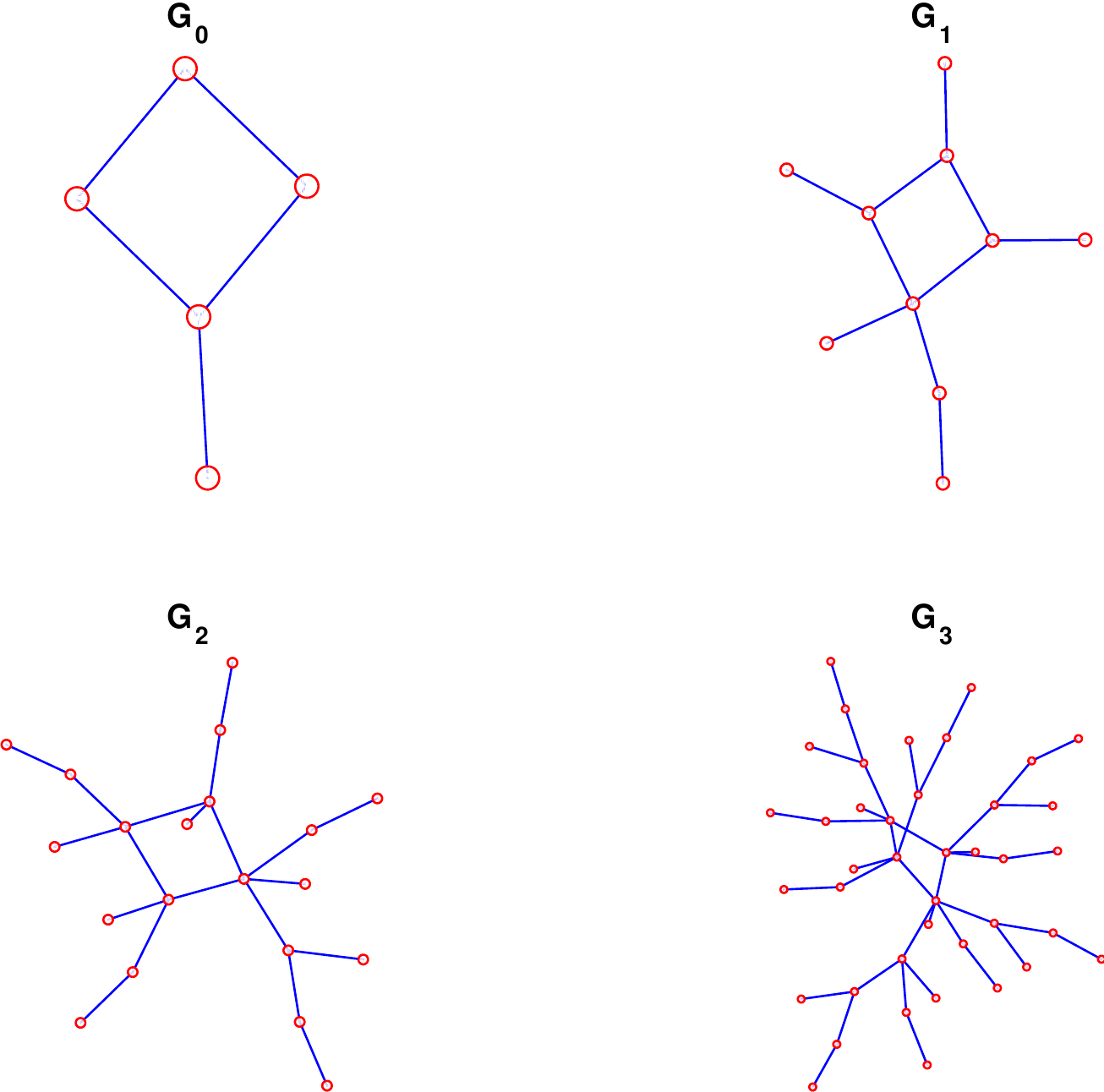}
  \caption{Whiskering a graph by adding a leaf to every node}
  \label{fig.whisker}
\end{figure}
Graph whiskering is a process for adding nodes to a graph, originally studied for the purpose of looking at monomial ideals~\cite{Biermann2013}.
For each whiskering iteration, a single unique node is connected to an already existing node in the graph, for all nodes in the graph.
This corresponds to concatenating the Laplacian in the following form:
\begin{align}
  L \longrightarrow \left[\begin{array}{cc}
      L + I&-I\\-I&I
    \end{array} \right] = L',\label{eq.whisker}
\end{align}
where we denote the operation as $L'=\mathscr{W}_1(L)$.
Figure~\ref{fig.whisker} shows an example of whiskering a graph three times.
This process is of great interest because of several properties that make it useful for control theoretic analysis, namely it preserves controllability and provides guarantees on the performance of control exerted on the resulting network.
The former property is captured in the following theorem.
\begin{theorem}\label{thm.whisker}
  Let $L'= \mathscr{W}_1(L)$. The pairs $(L',[b^T,b^T]^T)$ and $(L',[b^T,\mathbf{0}^T]^T)$ are controllable if and only if the pair $(L,b)$ is controllable.
\end{theorem}
\begin{proof}
  We prove the contrapositive using the Popov-Belevitch-Hautus (PBH) test~\cite{Mesbahi2010}. 
Suppose that $(L,b)$ is uncontrollable. 
Then, there exists $w$ such that $w^Tb=0$ and $w^TL=\lambda L$. 
We show by construction that there exists a left eigenvector of $L'$ that is orthogonal to the columns of $[b^T,b^T]^T$ and of $[b^T,\mathbf{0}^T]$.
We claim that $[ {w}^T,~ {\alpha}^T]^T$ is an eigenvector of $L'$ with eigenvalue $\Lambda$, where 
  \begin{align}
   {\alpha} &= \dfrac{1}{1-\Lambda} {w},~ \Lambda = \frac{1}{2} \left(\sqrt{\lambda ^2+4}+\lambda +2\right).
  \end{align}
The Laplacian is symmetric, and so its left eigenvectors are transposed right eigenvectors.
Therefore, a computation yields 
\begin{align}
 &\left[\begin{array}{cc}
      L + I&-I\\-I&I
    \end{array} \right] \left[
                    \begin{array}{c}
                       {w} \\   {\alpha}
                    \end{array} \right]= \left[
                                         \begin{array}{c}
                                           (L+I) {w} - I {\alpha}\\
                                           I {\alpha} - I {w}
                                         \end{array}\right]\\
&= \left[ 
  \begin{array}{c}
    (\lambda+1) {w} -  {\alpha}\\
 {\alpha} -  {w}
  \end{array}\right]=\left[
  \begin{array}{c}
    \Lambda {w}\\
{\Lambda} {\alpha}
  \end{array} \right].
\end{align}
This is orthogonal to the columns of $[b^T,b^T]^T$ and $[b^T,\mathbf{0}^T]^T$.

For the reverse direction, assume that $(L',b=[b_1^T,b_2^T]^T)$ is uncontrollable. 
Then by the PBH test, we have an eigenvector of $L'$ orthogonal to the columns of $b$
: 
\startcompact{small}
\begin{align}
 &L' = \lambda\left[
                    \begin{array}{c}
                       {w}_1\\ {w}_2
                    \end{array}\right]=  \left[\begin{array}{c}
      L {w}_1 + ( {w}_1- {w}_2)\\ {w}_2- {w}_1
    \end{array} \right]\\
  &[w_1^T,w_2^T]\left[
  \begin{array}{c}
    b_1\\b_2
  \end{array}\right] = \left[ 
  \begin{array}{c}
    w_1^Tb_1\\w_2^Tb_2
  \end{array}\right] = \left[ 
  \begin{array}{c}
    \mathbf{0}\\\mathbf{0}
  \end{array}\right].\label{eq.ortho}
\end{align}
 \stopcompact{small}
It follows that 
\begin{align}
w_2&=\dfrac{1}{1-\lambda}w_1,~Lw_1= \left(\dfrac{1-(1-\lambda)^2}{1-\lambda}\right)w_1,
\end{align}
and so $w_1$ is an eigenvector of $L$ with eigenvalue $(1-(1-\lambda)^2)(1-\lambda)^{-1}$.
It it is clear that from Equation~\eqref{eq.ortho}, $w_1^Tb_1=\mathbf{0}$, and so $(L,b_1)$ is uncontrollable.
This result also holds when $b_2=\mathbf{0}$, and the theorem follows.
\end{proof}
Theorem~\ref{thm.whisker} therefore provides a useful way of ensuring that a large graph is controllable.
Since rank controllability tests for very large graphs are often inaccurate due to machine precision, one can start with a small graph on which a controllability test is easily performed and iterate this process several times until a sufficiently large network is obtained.

A natural question to ask is what other types of growth processes preserve controllability?
For example, one can attach a leaf and a path of length 2 to every node (as shown in Figure~\ref{fig.pathwhisk}).
This corresponds to concatenating the Laplacian as follows: 
\begin{align}
  L\longrightarrow \left[
  \begin{array}{cccc}
    L+2I&-I&-I&\mathbf{0}\\
    -I&I&\mathbf{0}&\mathbf{0}\\
    -I&\mathbf{0}&2I&-I\\
    \mathbf{0}&\mathbf{0}&-I&I
  \end{array}\right]=L'.\label{eq.whiskerpath}
\end{align}
We denote $L'=\mathscr{W}_2(L)$. It turns out that this growth process also preserves controllability.
\begin{figure}
\centering
\includegraphics[scale=0.9]{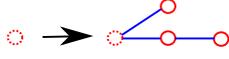}
\caption{Adding a leaf and a path of length 2 to every node}
\label{fig.pathwhisk}
\end{figure}

\begin{theorem}\label{thm.path}
 Let $L'= \mathscr{W}_2(L)$. Then, it follows that the pairs $(L',[b^T,b_i^T,b_i^T,b_i^T]^T)$ (where $b_i\in\{b,\mathbf{0}\}$) are controllable if and only if the pair $(L,b)$ is controllable.
\end{theorem}
\begin{proof}
We again prove the contrapositive using the PBH test.
Suppose that $(L,b)$ is uncontrollable.
Then, there exists $w$ such that $w^Tb=0$ and $w^TL=\lambda L$. 
We show by construction that there exists a left eigenvector of $L'$ that is orthogonal to the columns of $[b^T,b^T]^T$ and of $[b^T,\mathbf{0}^T]$.
It can be verified in a similar computation as in the proof of Theorem~\ref{thm.whisker} that $[w^T,\alpha^T,\beta^T,\gamma^T]$ is a left eigenvector of $L'$ with eigenvalue $\Lambda$, where 
 \startcompact{small}
\begin{align}
 \gamma = \dfrac{1}{1-\Lambda}\beta, ~ \beta = \left(\dfrac{1}{1-\Lambda}-\Lambda-2\right)^{-1}w, ~\alpha = \dfrac{1}{1-\Lambda}w,
\end{align}
 \stopcompact{small}
and where $\Lambda$ satisfies the equation $  \lambda +2=\frac{\Lambda  \left(\Lambda ^3-4 \Lambda +2\right)}{\Lambda ^3-2
   \Lambda +1}$.
Note that $\alpha$ and $\beta$ are simply scalings of $w$, and therefore $\gamma$ is also a scaling of $w$.
Therefore, $\alpha,\beta$ and $\gamma$ are all orthogonal to the columns of $b$: $\alpha^Tb = \beta^Tb = \gamma^Tb = w^Tb = \mathbf{0}$, and so it is clear that $ [w^T,\alpha^T,\beta^T,\gamma^T] [b^T,b_i^T,b_i^T,b_i^T]^T = \mathbf{0}$ for $b_i\in\{b,\mathbf{0}\}$.
It follows that  $(L',[b^T,b_i^T,b_i^T,b_i^T]^T)$ is uncontrollable.

For the reverse direction, assume that the pair $(L',b=[b_1^T,b_2^T,b_3^T,b_4^T]^T)$ is uncontrollable.
Then by the PBH test, there exists an eigenvector $w=[w_1^T,w_2^T,w_3^T,w_4^T]^T$ of $L'$ orthogonal to the columns of $b$ with eigenvalue $\lambda$ such that 
  \begin{align}
&\left[
  \begin{array}{cccc}
    L+2I&-I&-I&\mathbf{0}\\
    -I&I&\mathbf{0}&\mathbf{0}\\
    -I&\mathbf{0}&2I&-I\\
    \mathbf{0}&\mathbf{0}&-I&I
  \end{array}\right]\left[
                              \begin{array}{c}
                                w_1\\w_2\\w_3\\w_4
                              \end{array}\right] =\lambda\left[
                              \begin{array}{c}
                                w_1\\w_2\\w_3\\w_4
                              \end{array}\right].
\end{align}
It follows from a simple computation that  $w_1$ is an eigenvector of $L$ with eigenvalue 
\begin{align}
  \Lambda = \frac{\lambda ^4-6 \lambda ^3+14 \lambda ^2-14 \lambda +4}{(\lambda -1)
   \left(\lambda ^2-3 \lambda +1\right)}.
\end{align}
Since $w$ is orthogonal to the columns of $b$, it follows that $w_1^Tb_1=\mathbf{0}$, and the theorem follows.
\end{proof}
As Theorem~\ref{thm.path} suggests, there are many ways to grow graphs such that they remain controllable.
In this case, we have shown that adding a \emph{cluster} of nodes, namely a leaf and a path of length 2, to each node preserves controllability.
It is natural to examine what types of node clusters in general each node can be replaced with to preserve controllability.
The following theorem places some conditions on these clusters.
\begin{theorem}[General Graph Growth]\label{thm.concat}
Let $L$ be an $n\times n$ graph Laplacian, and let $L_\delta$ be $n\times n$, $C$ be $r\times r$ and $B$ be $n\times r$ (where $r=kn$ for $k\in\mathbb{Z}_+$) such that the matrix 
\begin{align}
  L'=\left[
    \begin{array}{cc}
      L+L_\delta&B\\B^T&C
    \end{array} \right]
\end{align}
is a graph Laplacian.
Since we are interested in adding the same cluster to each node, we can write $L_\delta = sI$, where $s$ is the number of edges added to the node when attaching it to the cluster.
We have the following results.

1) Let $b_n,w_1\in\mathbb{R}^n$ and $b_r,\beta\in\mathbb{R}^r$.
Suppose $(L',b=[b_1,b_r])$ is uncontrollable.  
Then, there exists an eigenvector $W_1\neq0$ such that $L'W_1=\Lambda w$, $W_1^Tb=\mathbf{0}$ where, say $W_1=[w_1^T,\beta^T]^T$. 
Then $(L,b_1)$ is uncontrollable when  $(\Lambda I -C)$ is invertible and when $w_1$ is an eigenvector of $B(\Lambda I -C)^{-1}B^T$.

2) Suppose $(L,b)$ is uncontrollable. Then, there exists $w\neq 0$ such that $Lw=\lambda w$ with $w^Tb=0$. We thus have that $(L',[b,\mathbf{0}])$ is uncontrollable if there exists $\Lambda\geq 0$ such that $(\Lambda I- C)^{-1}$ is invertible, and $w$ is an eigenvector of $B(\Lambda I -C)^{-1}B^T$ such that $B(\Lambda I -C)^{-1}B^Tw=f(\Lambda)$, where $\Lambda$ satisfies $\Lambda-\lambda-s=f(\Lambda)$.

\end{theorem}
\begin{proof}
  We prove the two results separately.

1) Assuming the notation in Theorem~\ref{thm.concat}-(1), suppose that $(L',b=[b_1,b_r])$ is uncontrollable.
Then, by the PBH test, there exists an eigenvector $W_1=[w_1^T,\beta^T]^T\neq0$ such that $L'W_1=\Lambda w$, $W_1^Tb=\mathbf{0}$ yielding:
\begin{align}
  L'W_1&=\left[
  \begin{array}{cc}
    (L+sI)w_1 + B\beta\\B^Tw_1 + C\beta
  \end{array}\right]=\Lambda \left[
                         \begin{array}{c}
                           w_1\\\beta
                         \end{array}\right].\label{eq.bet}
\end{align}
Therefore, if $(\Lambda I-C)$ is invertible, the lower entry of the vector in Equation~\eqref{eq.bet} gives $\beta = (\Lambda I -C)^{-1}B^T w_1$ and the first entry of the vector in Equation~\eqref{eq.bet} gives 
\begin{align}
  Lw_1 = \left[(\Lambda-s)I - B(\Lambda I - C)^{-1}B^T\right]w_1.
\end{align}
This equation admits $w_1$ as an eigenvector of $L$ if the action of $ B(\Lambda I - C)^{-1}B^T$ on $w_1$ is to scale $w_1$ by a fixed amount.
In other words, if $w_1$ is an eigenvector of $B(\Lambda I - C)^{-1}B^T$, then it is an eigenvector of $L$, and since $w_1^T b_1=\mathbf{0}$, the result follows.

2) Assuming the notation in Theorem~\ref{thm.concat}-(2), suppose that $(L,b)$ is uncontrollable. 
Then, there exists $w\neq 0$ such that $Lw=\lambda w$ with $w^Tb=0$.
We seek an admissible solution for the equation 
\begin{align}
  \left[
    \begin{array}{cc}
      L+L_\delta&B\\B^T&C
    \end{array} \right]\left[
                         \begin{array}{c}
                           w\\\beta
                         \end{array}\right]
&=\Lambda \left[
                         \begin{array}{c}
                           w\\\beta
                         \end{array}\right]\label{eq.vec1}
\end{align}
in terms of the eigenvalue $\Lambda$ of $L'$, and the lower part of the eigenvector, $\beta$.
If $(\Lambda I-C)$ is invertible, we can write $\beta = (\Lambda I-C)^{-1}B^Tw$.
From the upper entry of the vector in Equation~\eqref{eq.vec1}, we get the relation
\begin{align}
  (\Lambda-\lambda-s)Iw= B(\Lambda I-C)^{-1}B^Tw.
\end{align}
Then, if $w$ is an eigenvector of $B(\Lambda I-C)^{-1}B^T$, say $B(\Lambda I-C)^{-1}B^Tw = f(\Lambda)w$, then we get an equation for $\Lambda$: 
\begin{align}
  \Lambda - \lambda - s = f(\Lambda).
\end{align}
We add the stipulation that $\Lambda$ must be \emph{admissible}: $\Lambda\geq0$ for it to be a Laplacian eigenvalue.
Finally, it is clear that since $w^Tb=0$, we have that $[w^T,\beta^T][b^T,\mathbf{0}^T]^T=\mathbf{0}$.
\end{proof}

In the next section, we discuss using a similar graph growing approach to optimize graph performance.
We will also discuss bounds obtained using the submodularity theorems in \S\ref{sec.prelim} on the performance of graphs generated using the whiskering method.

\subsection{Optimization Algorithms: Adding Leaves}
In this section, we discuss optimization problems that are related to growing graphs.
In particular, we consider efficient addition of node clusters to a specific set of nodes in the graph.

Consider a connected graph $\graph$ and its Laplacian matrix $L_\graph$.
The second-smallest eigenvalue $\lambda_2(L_\graph)$ is a measure of how interconnected the graph is.
It is also an inverse measure of how long it takes for agents connected with graph $\mathcal{G}$ to achieve consensus by convergence to the agreement subspace.
A well-known algorithm by Ghosh and Boyd~\cite{Ghosh2006} adds edges between unconnected nodes in $\graph$ to maximize $\lambda_2(L_\graph)$.

The algorithm considers a set of candidate edges between unconnected nodes in $\graph$, and selects the $k$ candidate edges that maximize $\lambda_2(L_\graph)$.
For a set of $m$ candidate edges $l=\{i,j\}$, let $a_l$ be the vector with all-zero entries except $(a_{l})_j=1$ and $(a_{l})_j=-1$ when $\{i,j\}$ is a candidate edge.
The selection of $k$ candidate edges from this set can be encoded with a $\{0,1\}^m$-vector $x$, where $x_l=1$ if the algorithm selects candidate edge $l$, and zero otherwise.
The optimization problem is then written in terms of the individual Laplacians $a_la_l^T$ for each edge $l$ as follows:
{\small
\begin{equation*}
\begin{aligned}
& {\text{maximize}}
& & \lambda_2\left(L_\graph + \sum_{l=1}^m x_la_la_l^T \right)\\
& \text{subject to}
& & \mathbf{1}^Tx = k \\
& 
& & x\in\{0,1\}^m.
\end{aligned}
\end{equation*}
}%
The standard relaxation of this problem into a semidefinite program (SDP) is of the form
{\small
\begin{equation*}
\begin{aligned}
& {\text{maximize}}
& & s\\
& \text{subject to}
& & s(I-\mathbf{1}\mathbf{1}^T/n) \preceq L(x) \\
& 
& & \mathbf{1}^Tx = k \\
& 
& & 0\leq x\leq 1 \\
&
& & L(x) = L_\graph + \sum_{l=1}^m x_la_la_l^T.
\end{aligned}
\end{equation*}
}%

We present a modification of this algorithm whereby one wants to add \emph{nodes} to the graph $\graph$ in such a way that the graph grows in order to maximize $\lambda_2(\graph)$.
Suppose $\graph$ has $n$ nodes.
We want to choose one of these $n$ nodes to attach leaves to in order to `grow' the graph to maximize $\lambda_2(\graph)$. 
Recall that  $L_\graph[I]$ is the principal submatrix of $L_\graph$ obtained by deleting the rows and columns of $L_\graph$ corresponding to the elements in the set $I\setminus[m]$.
Let $L_{\text{tot}}$ denote the graph that has every node whiskered, as in Equation~\eqref{eq.whisker}.
We can write this problem as
{\small
\begin{equation}
\begin{aligned}
& {\text{maximize}}
& & \lambda_2(L_{\graph'}) \\
& \text{subject to}
& & L_{\graph'} \in \left\{ L_{\text{tot}}\big[[n]\cup\{i\}\big],~i\in\{n+1,\dots,2n\} \right\} \\
&
& & L_{\text{tot}}=\left[\begin{array}{cc}
      L_\graph + I&-I\\-I&I
    \end{array} \right].\label{prob.search}
\end{aligned}
\end{equation}
}%

This can be solved via exhaustive search over all possible whiskerings; however this becomes computationally intractable for large $n$.
We can relax this problem to a modified Ghosh-Boyd Max-$\lambda_2(\graph)$ SDP as follows.
Let $e_i$ denote the $i$th standard basis vector in $\mathbb{R}^n$.
Then, we introduce a single node into the system and create a set of $n$ candidate edges potentially connecting the new node to any pre-existing node in the graph.
The individual Laplacian for each candidate edge is $a_ia_i^T$, where $a_i\in\mathbb{R}^{n+1}$ is of the form $a_i = [e_i,-1]^T$.
The SDP relaxation is then
{\small
\begin{equation}
\begin{aligned}
& {\text{maximize}}
& & s \\
& \text{subject to}
& & s\left(I_{n+1}-\dfrac{(\mathbf{1}\mathbf{1}^T)_{n+1}}{n+1}\right) \preceq L(x) \\
&
& & \mathbf{1}^Tx = 1 \\
& 
& & 0\leq x\leq 1 \\
&
& & L(x) = L_\graph' + \sum_{l=1}^n x_la_la_l^T\\
&
& & L_\graph' = \left[ 
  \begin{array}{cc}
    L_\graph& \mathbf{0}_n \\ \mathbf{0}_n^T & 0
  \end{array}\right],~\mathbf{0}_n\in\mathbb{R}^n.
\end{aligned}\label{prob.1}
\end{equation}
}%

Note that $L[K]$ is positive-definite. 
We can also relate the inverse of this matrix to the \emph{controllability Gramian}~\cite{HudobaDeBadyn2015a,Chapman2015} $P$, which is a measure of the steady-state covariance of the agent states.
The matrix $P$ is the positive-definite solution of the Lyapunov equation 
\begin{align}
  -PL[2:n] - L[2:n]P^T = -I,
\end{align}
and is given by $P = \frac{1}{2} L[2:n]^{-1}$. 
Certain submodular functions of $P$ (with respect to edge-addition) have been studied in~\cite{Summers2014}.

The trace of $P$ can be interpreted as an average amount of energy expended to move the agent states around the controllable subspace, and therefore it is of interest to be able to bound the value of $\Tr P$ on the results of our algorithms. We do this using the supermodularity properties of $M$-matrices from Theorem~\ref{thm.supermod}.
 
\begin{theorem}\label{thm.bnd1}
  Let $L_1'$ denote the whiskering process in Equation~\eqref{eq.whisker} and  let $L_2'$ denote the whiskering process in Equation~\eqref{eq.whiskerpath}, where $L_1$ and $L_2$ are $n\times n$, and so $L_1'$ is $2n\times2n$ and $L_2'$ is $4n\times4n$.
  Let the controllability Gramians $P_1'$ of $L_1'$ and $P_2'$ of $L_2'$ be the respective solutions to 
  \begin{align}
     -P_1'L_1'[2:2n] - L_1'[2:2n]P_1'^T &= -I\\
     -P_2'L_2'[2:4n] - L_2'[2:4n]P_2^T &= -I.
  \end{align}
Then, 
\begin{align}
  \Tr P_1' \geq n + C_1~\text{and}~\Tr P_2' \geq 4n + C_2,
\end{align}
where $C_1,C_2$ are constants depending on $L_1,L_2$ respectively.
\end{theorem}
\begin{proof} 
We observe that the solution to 
  \begin{align}
     -P'L'[2:2n] - L'[2:2n]P'^T = -I
  \end{align}
is given by $P' = \frac{1}{2} L'[2:2n]^{-1}$.
From Theorem~\ref{thm.supermod}, using the fact that $L'$ is an $M$-matrix, we have that 
\begin{align}
  \Tr(P_1) &= \Tr(L_1'[2:2n]^{-1})\\
 &\geq \Tr(L_1'[2:n]^{-1})+\Tr(L_1'[n+1:2n]^{-1})\\
&=  \Tr([L_1[2:n]+I]^{-1}) + \Tr(I) = C_1+n,
\end{align}
where $C_1=\Tr([L_1[2:n]+I]^{-1})$ depends only on $L_1$.
The second result for $L_2$ follows from an identical calculation, noting that 
\begin{align}
\Tr\left[   \begin{array}{ccc}
   I&\mathbf{0}&\mathbf{0}\\
   \mathbf{0}&2I&-I\\
   \mathbf{0}&-I&I
  \end{array}\right]^{-1} = 2\Tr (I) + \Tr(2I) = 4n.
\end{align}
\end{proof}
We can use this result to bound the trace of the controllability Gramian when adding a single node to the system. 
\begin{theorem}\label{thm.p}
  Consider the task of attaching a single node to the system with $n\times n$ Laplacian $L$ to maximize $\lambda_2$, as denoted in Problem~\eqref{prob.search}.
Let $L'$ be the subsequent Laplacian, and so $P=\frac{1}{2}L'[2:n+1]^{-1}$.
Then, $\Tr P \geq C + 1$, where $C$ is a constant depending only on $L$.
\end{theorem} 
\begin{proof}
  Using Theorem~\ref{thm.supermod} we can compute 
  \begin{align}
    \Tr(P) &= \Tr(L'[2:n+1]^{-1}) \\
&\geq \Tr(L'[2:n]^{-1}) + \Tr(L'[n+1:n+1]^{-1})\\
&= \Tr([L[2:n]+e_ie_i^T]^{-1}) + 1 \geq C+1,
  \end{align}
where $i$ is the index of the attachment node chosen, and $C=\min_i(\Tr([L[2:n]+e_ie_i^T]^{-1}))$ is a constant depending only on $L$.
\end{proof}

In the next section, we will consider adding a cluster, and provide a similar result on the performance $\Tr( P)$.
\subsection{Optimization Algorithms: Adding Clusters}
In the previous section, we considered the problem of optimally adding leaves to some nodes to optimize the algebraic connectivity of the graph.
We will now consider the problem of adding a cluster of a node and a length-2 path, as depicted in Figure~\ref{fig.pathwhisk}.                   

Let $\mathbf{0}_n\in\mathbb{R}^n$ and define  $a_{3\to4}=[\mathbf{0}_n^T,0,1,-1]^T$, $a_{i,1}=[e_i^T,-1,0,0]^T$ and $a_{i,2}=[e_i^T,0,-1,0]^T$.
Then, choosing an attachment node to maximize $\lambda_2$ can be written as
{\small
\begin{equation}
\begin{aligned}
& {\text{maximize}}
& & \lambda_2(L_{\graph'}) \\
& \text{subject to}
& & L_{\graph'} \in \big\{ L_{\text{tot}}\big[[n]\cup\{i,i+n,i+2n\}\big],\\
&
& &~i\in\{n+1,\dots,2n\} \big\} \\
&
& & L_{\text{tot}}=\left[
  \begin{array}{cccc}
    L+2I&-I&-I&\mathbf{0}\\
    -I&I&\mathbf{0}&\mathbf{0}\\
    -I&\mathbf{0}&2I&-I\\
    \mathbf{0}&\mathbf{0}&-I&I
  \end{array}\right].\label{prob.search2}
\end{aligned}
\end{equation}
}%
We can write the SDP relaxation as:
{\small
\begin{equation}
\begin{aligned}
& {\text{maximize}}
& & s \\
& \text{subject to}
& & s\left(I_{n+3}-\dfrac{(\mathbf{1}\mathbf{1}^T)_{n+3}}{n+3}\right) \preceq L(x) \\
&
& & \mathbf{1}^Tx = 1 \\
& 
& & 0\leq x\leq 1 \\
&
& & L(x) = L_\graph' + a_{3\to4}a_{3\to4}^T  \\
&
& & ~~~~~~+\sum_{l=1}^n x_l(a_{i,1}a_{i,1}^T +a_{i,2}a_{i,2}^T )\\
&
& & L_\graph' = \left[ 
  \begin{array}{cc}
    L_\graph& \mathbf{0}_{n\times 3} \\ \mathbf{0}_{3\times n} & \mathbf{0}_{3\times 3}
  \end{array}\right].
\end{aligned}\label{prob.2}
\end{equation}
}%

Lastly, we provide a performance bound on the Gramian analogous to Theorem~\ref{thm.p}. 
\begin{theorem}
   Consider the task of attaching a single node to the system with $n\times n$ Laplacian $L$ to maximize $\lambda_2$, as denoted in Problem~\eqref{prob.search2}. 
Let $L'$ be the subsequent Laplacian, and so $P=\frac{1}{2}L'[2:n+3]^{-1}$.
Then, $\Tr P \geq C + 4$, where $C$ is a constant depending only on $L$.
\end{theorem}
\begin{proof}
    \begin{align}
    \Tr(P) &= \Tr(L'[2:n+3]^{-1}) \\
&\geq \Tr(L'[2:n]^{-1}) + \Tr(L'[n+1:n+3]^{-1})\\
&= \Tr([L[2:n]+2e_ie_i^T]^{-1}) + 4 \geq C+4,
  \end{align}
where $i$ is the index of the attachment node chosen, and $C=\min_i(\Tr([L[2:n]+2e_ie_i^T]^{-1}))$ is a constant depending only on $L$.
\end{proof}

\section{ALGORITHM IMPLEMENTATION}\label{sec.alg}
In this section, we show examples of the optimization problems discussed in the previous section.

The optimization problems~\eqref{prob.search} and \eqref{prob.search2} were implemented using \texttt{cvx}~\cite{cvx,Grant2008}.
An additional relaxation method used to solve problems~\eqref{prob.1} and~\eqref{prob.search2} discussed in~\cite{Ghosh2006}, known as the \emph{perturbation heuristic}, was also implemented for the purpose of comparison.
At each iteration of the heuristic algorithm, the node cluster is attached to one node chosen by selecting the node with the largest value of $(v_i - v_{n+1})^2$, where $v$ satisfies $L'v = \lambda_2 v$.
Here, $v_{n+1}$ is the entry of $v$ corresponding to the node in the cluster attaching the cluster to node $i$.

If there is more than one node attaching the cluster to the node in the graph, then without loss of generality, denote these (say, $l$) nodes as $n+1,\dots,n+l$.
Then, the perturbation heuristic is to find the node $i\in[n]$ maximizing $\sum_{j=1}^l(v_i-v_{n+j})^2$ at each iteration.

\begin{figure}
  \centering
  \includegraphics[width=0.75\columnwidth]{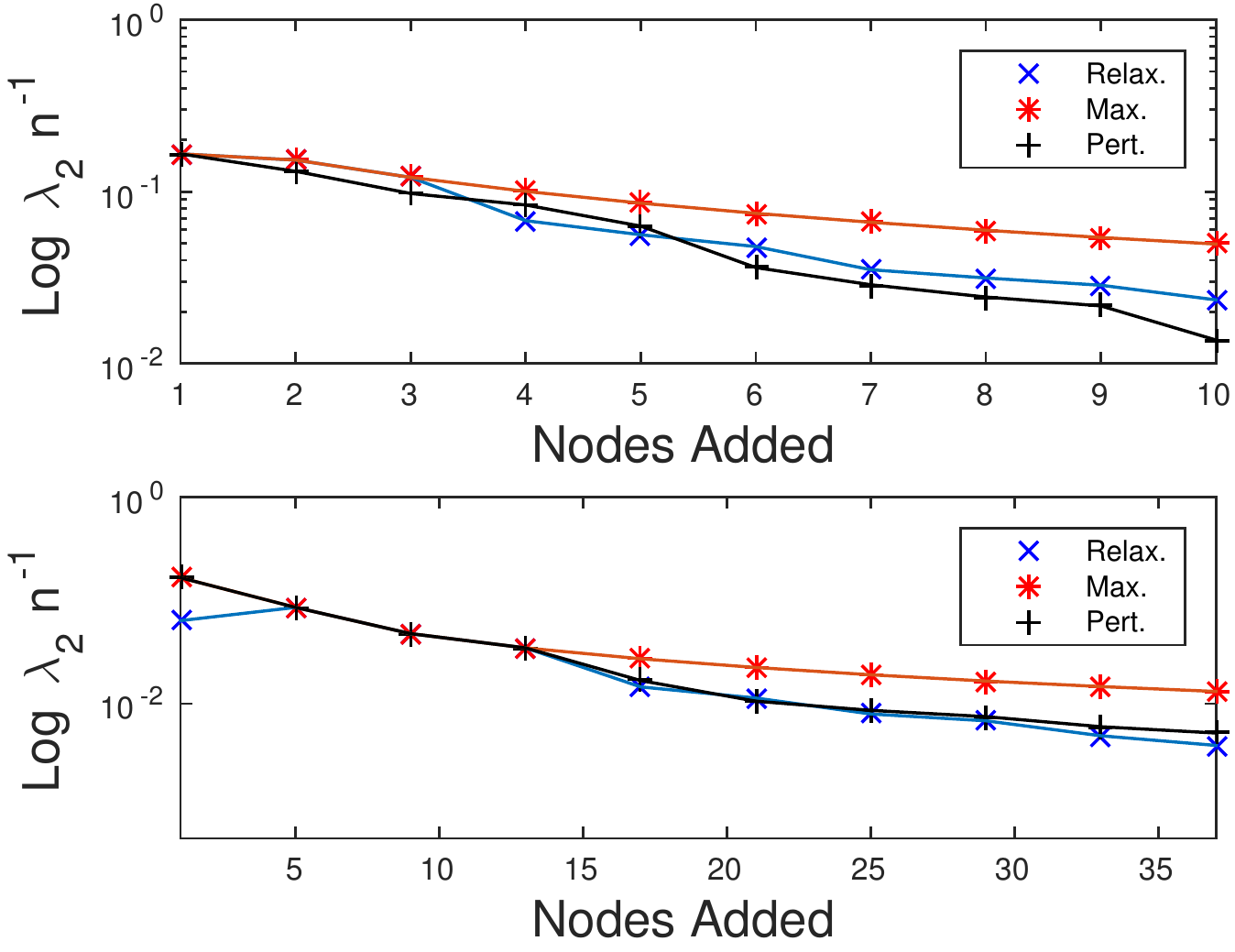}
  \caption{Max-$\protect\lambda_2$ algorithm results for adding leaves (top) and path clusters (bottom), using the convex relaxation ($\protect\times$), exhaustive search (\protect\textasteriskcentered) and perturbation heuristic ($+$).}
\label{fig.alg}
\end{figure}
\begin{figure}
  \centering
  \includegraphics[width=0.50\columnwidth]{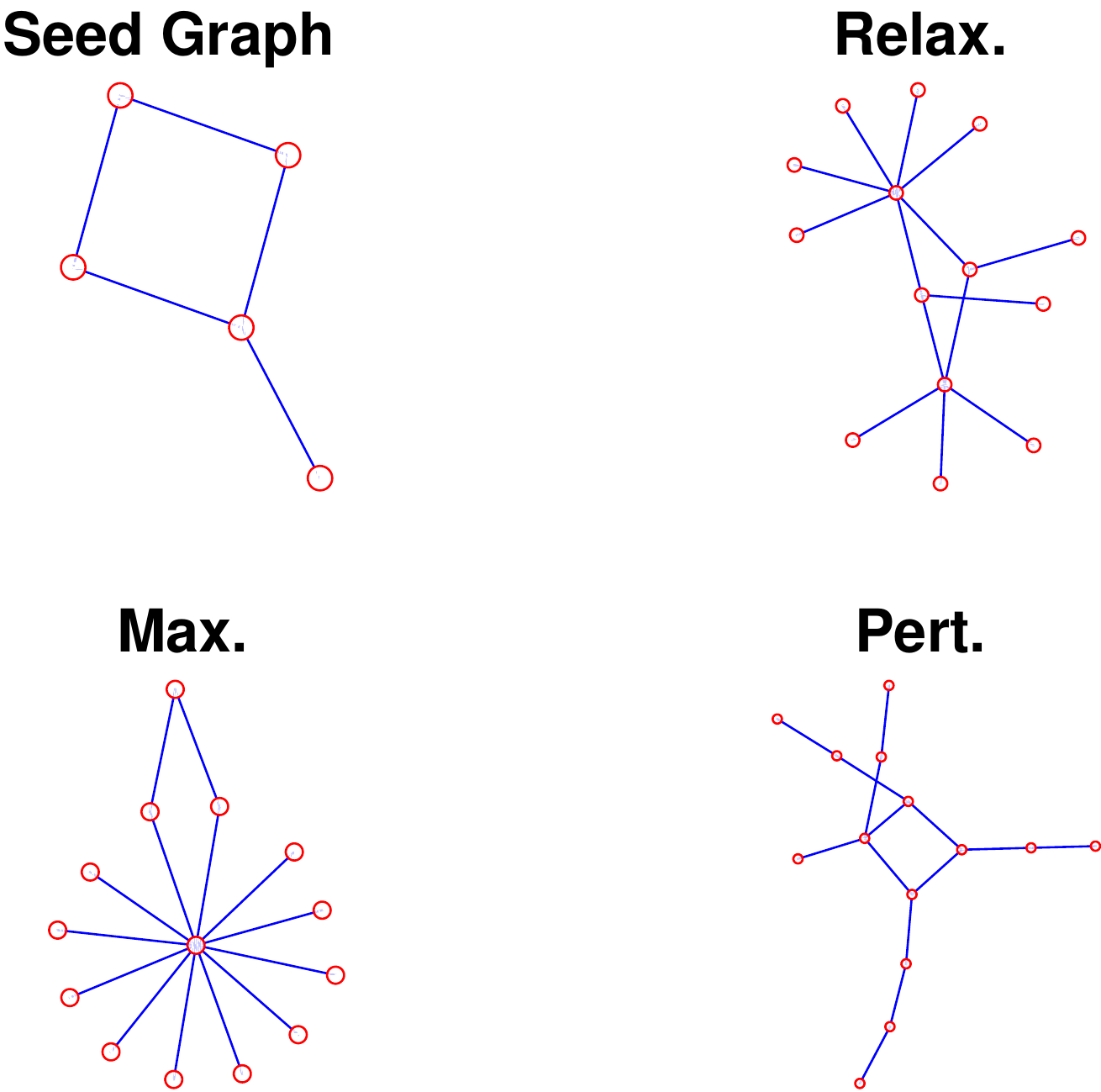}
\caption{Seed graph, and final graph after 9 iterations of the leaf-adding problem using the SDP relaxation~\protect\eqref{prob.1}, exhaustive search over problem~\protect\eqref{prob.search} and the perturbation heuristic.}
\label{fig.leaf}
\end{figure}
\begin{figure}
  \centering
  \includegraphics[width=0.50\columnwidth]{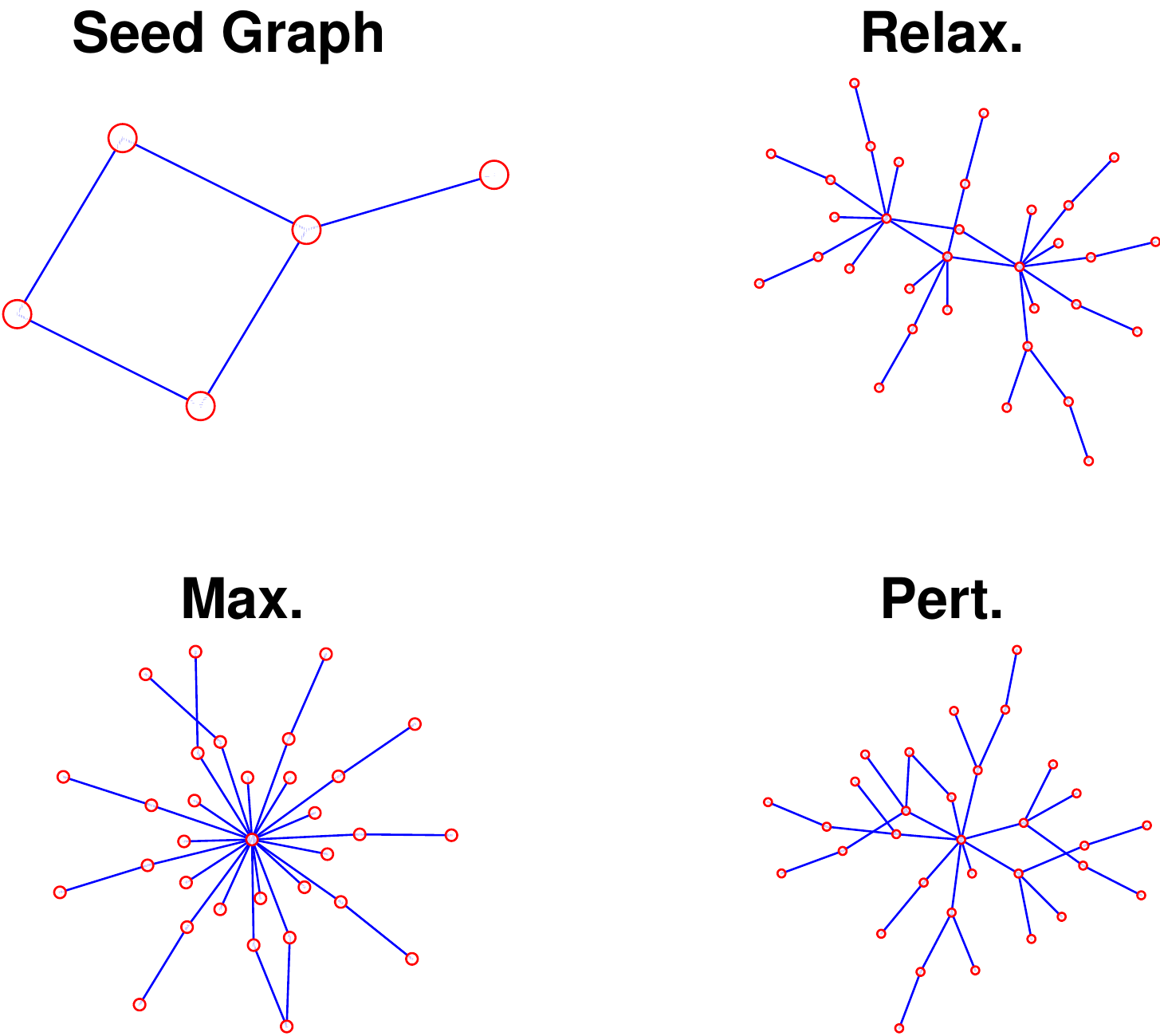}
\caption{Seed graph, and final graph after 9 iterations of the path-cluster-adding problem using the SDP relaxation~\protect\eqref{prob.2}, exhaustive search over problem~\protect\eqref{prob.search2} and the perturbation heuristic.}
\label{fig.path}
\end{figure}

The results of running these algorithms for 9 iterations are shown in Figures~\ref{fig.alg}, \ref{fig.leaf}, and \ref{fig.path}.
The seed graphs, and final graphs after 9 iterations for each of the three techniques (exhaustive search, convex relaxation and perturbation heuristic) are shown in Figure~\ref{fig.leaf} for adding a single leaf, and in Figure~\ref{fig.path} for adding the path cluster.
For both cases, the convex relaxations (problems~\eqref{prob.1} and~\eqref{prob.2}) perform reasonably well and pick out slightly suboptimal solutions, as seen in Figure~\ref{fig.alg}.
\section{CONCLUSIONS AND FUTURE WORKS}\label{sec.conc}
In this paper, we explored methods of constructing graphs by iterating a procedure that preserves controllability.
In Theorem~\ref{thm.whisker}, we showed that adding a leaf to every node preserves controllability, and in Theorem~\ref{thm.path} we showed that adding a cluster of two paths (of length 1 and 2) to each node also preserves controllability. We provided bounds on the performance of the resulting graph using submodularity.
We obtained general conditions for preserving controllability under iterative graph growing in Theorem~\ref{thm.concat}.

An interesting area of further work is to classify all graph clusters that one can attach to all nodes in the network that preserve controllability; in other words describe the matrix $L'$ appearing in Theorem~\ref{thm.concat} in terms of graph objects.
It would also be worth exploring the combination of the node-addition algorithms presented in the paper with edge-adding algorithms, for example maximizing $\lambda_2$~\cite{Ghosh2006} and adding edges to make cycles for more robust consensus~\cite{Zelazo2013}.

\section{ACKNOWLEDGMENTS}

MHdB thanks Airlie Chapman for useful conversations on submodular optimization.



\end{document}